\documentclass{amsart}

\usepackage{latexsym, amssymb, amscd, amsthm, amsxtra, amsmath,amsthm }
\usepackage{graphics, graphicx, color}
\usepackage{ifpdf}
\usepackage[format=hang,indention=-1cm,small]{caption}
\usepackage[caption=false]{subfig}
\usepackage{multirow}
\usepackage{hyperref}
\usepackage{mathrsfs}
\usepackage[english]{babel}
\usepackage{pst-node}
\usepackage{amsmath,amscd}
\usepackage[shortlabels]{enumitem}
\usepackage{multirow}
\usepackage{array}
\newtheorem{thm}{Theorem}[section]

\newtheorem{lem}[thm]{Lemma}

\theoremstyle{definition}

\newtheorem{ques}[thm]{Question}
\theoremstyle{remark}

\newcommand{\A}{\mathbb{A}}
\newcommand{\PP}{\mathbb{P}}
\newcommand{\FF}{\mathbb{F}}

\newcommand{\ZZ}{\mathbb{Z}}

\newcommand{\al}{\alpha}

\newcommand{\cO}{\mathcal{O}}

\newcommand{\fm}{\mathfrak{m}}
\newcommand{\fn}{\mathfrak{n}}

\newcommand{\orp}{\operatorname{ord}_p}
\newcommand{\GL}{\operatorname{GL}}
\newcommand{\End}{\operatorname{End}}
\newcommand{\Aut}{\operatorname{Aut}}
\newcommand{\spec}{\operatorname{Spec}\,}

\title{Integral Periodic Orbits on Affine Spaces}
\author{Minchan Kang}
\address{Department of Mathematical Sciences, Seoul National University, Gwanak-ro 1, Gwankak-gu, Seoul, South Korea 08826}
\email{azboy@snu.ac.kr}
\date{\today}
\begin{document}
\maketitle
\begin{abstract}
In this paper, we give an elementary proof on the existence of an effective uniform upper bound on the size of integral periodic orbits of a single endomorphism in an affine space, dependent solely on its dimension. In fact, we derive a formula relating the primitive period to the local primitive period obtained through reduction modulo prime number. In particular, we prove that the size of any integral periodic orbit in the affine plane does not exceed 24.
\end{abstract}
\section{Introduction}
We investigate dynamical systems arising from polynomial maps with integer coefficients. Let $X$ be a set, and consider an endomorphism $f:X\to X$. An element $P\in X$ is said to be \emph{$f$-periodic} if there exists a positive integer $n\geq1$ such that $f^n(P)=P$. If $P$ is $f$-periodic, then the $f$-orbit
$$\cO_f(P)=\{f^k(P):k\geq0\}$$
is a finite set, and the size $|\cO_f(P)|$ is the smallest integer $n\geq1$ such that $f^n(P)=P.$ We call this integer the \emph{primitive period of $P$}. We are particularly interested in the set $\A^N(\ZZ)$ of $\ZZ$-points of an affine $N$-space for an integer $N$, with an endomorphism $f:\A^N\to\A^N$ defined over $\ZZ$.

A natural question one can think of is whether there exists an effectively computable constant, depending solely on $N$, which bounds the size of any periodic orbit in $\A^N$. It is a well-known fact, which can be readily shown, that the size of any integral periodic orbit of an affine line is at most 2. Indeed, for any $x, y \in \ZZ$, $x - y$ divides $f(x) - f(y)$ for any polynomial $f(t) \in \ZZ[t]$. In fact, for higher dimensional affine spaces, or more generally for an arbitrary separated scheme of finite type over $\ZZ$, the existence of a universal constant can be deduced by Fakhruddin's work \cite{Fak}. Moreover, Whang demonstrated the existence under more general setting, a set $S$ of endomorphisms and a generalized definition of $S$-periodic orbits \cite{whang}.

In this paper, we present a more elementary proof of the existence than the one given in \cite[Theorem 2]{Fak} or \cite[Theorem 2.3]{whang}, in the sense that it doesn't use any scheme-theoretic argument, with respect to our particular focus on a single endomorphism on $\A^N$. Our main method is analyzing the relationship with a local period obtained through reduction modulo prime number. This method was used in \cite{Li,MS1,MS2,W1989,T1994,zieve} for rational functions in $\PP^1$ having good reduction, and was generalized in \cite{Hutz} for endomorphisms on smooth projective varieties having good reduction. Note that we cannot apply these results directly, as we are not working within a proper scheme, and we are not exclusively considering primes having good reduction.

We prove the following Theorem:
\begin{thm}\label{localglobal}
    Let $f:\A^N\to\A^N$ be an endomorphism of $\A^N$ defined over $\ZZ$. Let $p$ be a prime number, and let $\tilde{f}$ be the reduction modulo $p$ map of $f$. If $P\in\A^N(\ZZ)$ is a periodic point of $f$ with primitive period $n$, then
    $$n=mdp^e,$$
    where $m$ is the primitive period of $\tilde{P}$, the reduction modulo $p$ point of $P$, in $\A^N(\FF_p)$, $d$ is a positive integer less than or equal to $p^N-1$, and $e$ is a nonnegative integer.
\end{thm}
Throughout this paper, we denote by $\orp$ the $p$-adic valuation of $\ZZ$, normalized as $\orp(p)=1$, for each prime $p$. By applying the theorem above for $p=2$ and $p=3$, we can directly establish an upper bound of $\orp(n)$ for any prime number $p$, which is zero for all but finitely many $p$. This gives an effective upper bound on the size of any integral periodic orbit on $\A^N$, which depends only on $N$. Moreover, applying the result from \cite{Fak} in our formula, we can additionally obtain a better upper bound. In particular, for $N=2$, we prove that any integral periodic orbit has a size less than or equal to 24.
\begin{thm}\label{affineplane}
    Any size of an integral periodic orbit of $\A^2$ is at most 24.
\end{thm}
Let us note that the bounds discussed in this paper may not be optimal. It is an interesting question to obtain an optimal value, as knowing the optimal upper bound provides an efficient algorithm for determining periodicity when a point and an endomorphism are given. Currently, there is no known example of an integral periodic orbit in $\A^2$ with a size larger than 6. Hence, one might expect that the maximal length of the periodic orbit emerges in the linear case.
\begin{ques}
    Can the size of an integral periodic orbit in $\A^N$ exceed the maximal order in $\GL_{N+1}(\ZZ)$?
\end{ques}
\subsection*{Acknowledgements}
I would like to thank to Professor Junho Peter Whang for suggesting the problem and offering valuable guidance throughout this project. This work was supported by 2023 summer Undergraduate Research Internship from College of Natural Sciences, Seoul National University.
\section{Reduction Modulo Prime}
In this section, we prove Theorem \ref{localglobal}. We first prove the following Lemma, which is a slight generalization of \cite[Corollary 2]{glnfp}.
\begin{lem}\label{matrixlem}
    Let $N$ be a positive integer and $p$ a prime number. Let $A$ be a $N\times N$ matrix with coefficients in $\FF_p$. Let $g = g(A)$ be the smallest positive integer such that $A^g$ is diagonalizable over $\FF_p$ and has only 0 and 1 as eigenvalues. Then
    $$g\leq p^N-1.$$
\end{lem}
\begin{proof}
If $A$ has 0 as its unique eigenvalue, then by Cayley-Hamilton theorem, $A^N=0$, and we are done. Otherwise, again by Cayley-Hamilton theorem, $\FF_p[A]$ has at most $p^N-1$ nonzero elements. Hence, there exists a positive integer $k<p^N$ such that
\begin{equation}\label{kl}
    A^k=A^{p^N}.
\end{equation}
Let $t$ be the smallest positive integer such that
$$t(p^N-k)\geq k.$$
Then $(t-1)(p^N-k)<k$, so $t(p^N-k)<p^N$. Now by \eqref{kl},
$$A^k(A^{p^N-k}-1)=0,$$
so
$$A^k(A^{t(p^N-k)}-1)=0$$
and
$$A^{t(p^N-k)}(A^{t(p^N-k)}-1)=0.$$
Hence $A^{t(P^N-k)}$ is diagonalizable over $\FF_p$ with eigenvalues 0 and 1.
\end{proof}
Now we prove Theorem \ref{localglobal}.
\begin{proof}[Proof of Theorem \ref{localglobal}]
    Let $P=(a_1,\ldots,a_N)\in\A^N(\ZZ)$ be a periodic point of $f$ with primitive period $n$. By conjugating
    $$((x_1,\ldots,x_N)\mapsto(x_1-a_1,\ldots,x_N-a_N))\in\Aut(\A^N/\ZZ),$$
    we may assume $P=(0,\ldots,0)$. Let $\tilde{P}\in\A^N(\FF_p)$ be a reduction modulo $p$ point of $P$. 
     Since
    $$\tilde{f}^n(\tilde{P})=\widetilde{f^n(P)}=\tilde{P},$$
    $\tilde{P}$ is a periodic point of $\tilde{f}$. If $m$ is the primitive period of $\tilde{P}$ in $\A^N(\FF_p)$, then $m$ divides $n$ by the minimality of $m$. Hence replacing $f$ by $f^m$, we may assume $m=1$.

    If $n=1$, we are done. Assume $n>1$, and for any $k\geq1$, let
    $$f^k(P)=f^k(0,\ldots,0)=(A_1^{(k)},\ldots,A_N^{(k)}).$$
    Since we are assuming $\tilde{f}(\tilde{P})=\tilde{P}$, we see that $p$ divides $A_i^{(1)}$ for all $i=1,2,\ldots,N$. Let
    $$r=\min_i\{\orp(A_i^{(1)})\}\geq1.$$
Let $f=(f_1,\ldots,f_N)$. As $A_i^{(1)}$ are constant terms of $f_i$ for each $i=1,\ldots,N$, we have $\orp(A_i^{(2)})\geq r$ for all $i=1,\ldots,N$. Hence $\min_i\{\orp(A_i^{(2)})\}\geq r$, and inductively,
    \begin{equation}\label{order}
        \min_i\{\orp(A_i^{(k)})\}\geq r
    \end{equation}
    for any $k\geq1$.
Let
        $$D=D(f)=\left(
            \frac{\partial{f_i}(0,\ldots,0)}{\partial{x_j}}
        \right)_{1\leq i,j\leq N}\in\operatorname{Mat}_{N}(\ZZ).$$
By \eqref{order}, for $k\geq1$,
        $$\begin{pmatrix}
            A_{1}^{(k+1)}\\\vdots\\A_N^{(k+1)}
        \end{pmatrix}\equiv\begin{pmatrix}
            A_{1}^{(1)}\\\vdots\\A_{N}^{(1)}
        \end{pmatrix}+D\begin{pmatrix}
            A_{1}^{(k)}\\\vdots\\A_{N}^{(k)}
        \end{pmatrix}\pmod{p^{2r}}$$
since higher terms are divided by $p^{2r}$. Considering elements of $\A^N(\ZZ)$ as vectors in $\ZZ^n$, we inductively have
    $$f^{k+1}(P)\equiv(I+D+\cdots+D^k)f(P)\pmod{p^{2r}}$$
    for $k\geq1$. In particular,
    $$0=f^n(P)\equiv(I+D+\cdots+D^{n-1})f(P)\pmod{p^{2r}}.$$
Let $A_i^{(1)}=p^r\al_i$ for $i=1,\ldots,N$. Let $v=(\al_1,\ldots,\al_N)^T$ be the nonzero vector in $\ZZ^N$, whose reduction modulo $p$ is also nonzero vector in $\FF_p^N$ by definition of $r$. Then
    $$0\equiv(I+D+\cdots+D^{n-1})v\pmod{p^{r}}.$$
    In particular,
    \begin{equation}\label{eq1}
        0=(I+\tilde{D}+\cdots+\tilde{D}^{n-1})\tilde{v}
    \end{equation}
as a vector in $\FF_p^N$, where $\tilde{D}$ and $\tilde{v}$ are reduction modulo $p$ matrix and vector of $D$ and $v$, respectively. Multiplying $(I-\tilde{D})$ to the both side of \eqref{eq1}, we have
$$\tilde{D}^n\tilde{v}=\tilde{v}.$$
Let $d\geq1$ be the minimal integer such that
$$\tilde{D}^d\tilde{v}=\tilde{v}.$$

\noindent\textbf{Case 1.} If $d=1$, then \eqref{eq1} implies that $n\tilde{v}=0$. Hence $p$ divides $n$, since $\tilde{v}$ is a nonzero vector in $\FF_p^N$.

\noindent\textbf{Case 2.} If $d>1$, then $d$ divides $n$ by the minimality of $d$. The above equality implies that $\tilde{D}^d$ has eigenvector $\tilde{v}$ with eigenvalue 1. Let $g=g(\tilde{D})$ be the integer related to $\tilde{D}$ in Lemma \ref{matrixlem}. Then $d$ divides $g$.

Now, replace $f$ by $f^p$ and $f^d$ for Case 1 and Case 2, respectively, and apply the above process iteratively. It ends after a finite number of iterations since $n$ is finite. Note that both the matrix $D$ and the vector $v$ change during the replacement process. However
    $$\widetilde{D(f^k)}=\widetilde{D(f)}^k$$
for any $k\geq1$ since $m=1$, so replacing $\tilde{D}$ is same as taking power of $\tilde{D}$. Moreover, after raising $\tilde{D}$ to a suitable power $d_0$ during the process, where $d_0$ divides $g$, the situation eventually becomes restricted to Case 1. Thus
$$n=d_0p^e$$
for some $d_0$ dividing $g\leq p^N-1$ and $e\geq0$.
\end{proof}
As we have mentioned, Theorem 1.1 directly proves the existence of an effective uniform upper bound on the size of an integral periodic orbit.
\begin{thm}
    Let $N\geq1$ be an integer. There exists an effective constant $C(N)>0$ such that, for any endomorphism $f:\A^N\to\A^N$ defined over $\ZZ$ and a periodic point $P\in\A^N(\ZZ)$ of $f$, the size $|\cO_f(P)|$ of the orbit is less than or equal to $C(N)$.
\end{thm}
\begin{proof}
    Take any $f\in\End(\A^N/\ZZ)$ and let $n$ be a size of an integral periodic orbit of $f$. By applying Theorem \ref{localglobal} with $p=2$, we have
    $$n=md2^e,$$
    where $m,d,e$ are integers such that $m\leq 2^N$, $d\leq 2^N-1$, and $e\geq0$.
    Then $\operatorname{ord}_q(n)$ is bounded for any prime $q$, except for $2$. Moreover, prime numbers greater than $2^N$ cannot appear in the factorization of $n$. Applying Theorem \ref{localglobal} with $p=3$, we can also bound $\operatorname{ord}_2(n)$. If $N$ is given, we can effectively calculate each bound of $\orp(n)$ for any prime number $p$.
\end{proof}
\section{A Tighter Upper Bound}

In this section, we focus on constructing a lower value of the bound. We first recall the result of Fakhruddin.
\begin{thm}[{\cite[Proposition 2.2]{Fak}}]\label{Fak}
    Let $p$ be a prime number. Let $(A,\fm)$ be a local sub-$\ZZ_p$-algebra of $(\ZZ_p)^{p^n}$ of rank $p^n$ which is preserved by the automorphism $\sigma$ given by cyclic permutation of the coordinates. Furthermore, assume that $\sigma$ acts trivially in $\fm/\fm^2$. Then $n=0$ if $p\not=2$ and $n\leq 1$ if $p=2$.
\end{thm}
Applying Theorem \ref{Fak} for $p=2$, we can obtain an additional restriction to the size of integral periodic orbit.
\begin{thm}\label{primefactor}
    Let $N$ be an integer. Let $p(N)$ be the largest prime number less than $2^N$. Then any size of an integral periodic orbit of $\A^N$ divides
    $$2^{2N}\prod_{\substack{p : \text{prime}\\ 3\leq p\leq p(N)}}p^{2[N\log_p2]}.$$
\end{thm}
\begin{proof}
    Take $f\in\End(\A^N/\ZZ)$, and let $P\in\A^N(\ZZ)$ be a periodic point of $f$ with primitive period $n$. By Theorem \ref{localglobal} with $p=2$, we have
    $$n=md2^e$$
    for local period $m\leq 2^N$, a nonnegative integer $d\leq 2^N-1$, and an integer $e\geq0$. Hence for prime numbers $p\not=2$, we obtain the result.
    
    Replacing $f$ by $f^m$, we may assume $\tilde{P}$ is a fixed point of $\tilde{f}$.  Let $d=d_12^{e'}$, where $d_1$ is an odd integer. Again replacing $f$ by $f^{d_1}$, we may assume that $P$ has primitive period $2^j=2^{e+e'}$. Now consider $f$ as an endomorphism defined over $\ZZ_2$ and $P$ as a $\ZZ_2$-point of $\A_{\ZZ_2}^n$.
    
    Let
    $$Z=\bigcup_{i=0}^{2^j-1}f^i(P)(\spec\ZZ_2)$$
    be a closed subscheme of $\A_{\ZZ_2}^n$ endowed with the reduced closed subscheme structure. Then $Z=\spec A$ for some local sub-$\ZZ_2$-algebra $(A,\fm)$ of $(\ZZ_2)^{2^j}$ rank $2^j$. Moreover, $f$ induces an automorphism of $A$ given by cyclic permutation of the coordinates. The map
    $$f_{Z,*}:\fm/\fm^2\to\fm/\fm^2$$
    is a restriction of $f_*:\fn/\fn^2\to\fn/\fn^2$, where $\fn$ is the maximal ideal of a stalk of the closed point of $Z$ in $\A_{\ZZ_2}^n$. Note that $f_*$ is represented by
    $$\begin{pmatrix}
        1&0\\0&\tilde{D}
    \end{pmatrix}$$
    as a $\FF_2$-linear map, where $\tilde{D}$ is the $N\times N$ matrix with coefficients in $\FF_2$ defined in the proof of Theorem \ref{localglobal}. Let $h=2^{e_0}h'$ be the least integer such that $(f_{Z,*})^h$ is the identity map. Then $h$ divides $g=g(\tilde{D})$ in Lemma \ref{matrixlem}, so $e_0$ is smaller than $N$.
    
    Replace $f$ by $f^h$, and let $k=\max\{0,j-e_0\}$. Then $P$ has primitive period $2^k$, and $A$ is a local sub-$\ZZ_2$-algebra of $(\ZZ_2)^{2^k}$ of rank $2^k$ satisfying the assumption of Theorem \ref{Fak}, so $k\leq 1$. Hence,
    \[
        \operatorname{ord}_2(n)=\operatorname{ord}_2(m)+j
        \leq N+e_0+k\leq 2N-1+k\leq2N,
    \]
    which proves the assertion for $p=2$.
\end{proof}
By analyzing the integers that appear in the proof of Theorem \ref{primefactor}, we can derive more improved upper bound. In particular, we prove Theorem \ref{affineplane}.
\begin{proof}[Proof of Theorem \ref{affineplane}]
    Let $n$ be a size of an integral periodic orbit of $\A^2$. We have shown that
    $$n=md2^e$$
    for some $m\leq 4$ and $d\leq 3$. Recall that $d$ divides an integer $g\leq 2^2-1=3$. We use $j$, $h$ and $e_0$ as the same notation in the proof of Theorem \ref{primefactor}. 
    \begin{enumerate}
        \item[(i)] If $d=3$, then $g=3$ and $e_0=0$ since $h$ divides $g$. Hence $e=j\leq k\leq1$, and $n\leq 4\cdot3\cdot2=24$.
        \item[(ii)] Otherwise, $e_0\leq1$, and
        $$j\leq k+e_0\leq 2,$$
        so $n\leq 16$.
    \end{enumerate}
    In any cases, we can conclude that $n\leq 24$.
\end{proof}
\bibliographystyle{abbrv}
\bibliography{reference}
\end{document}